\documentclass[11pt, oneside, a4paper]{article}

\usepackage[a4paper,margin=1.4in]{geometry}
\usepackage[utf8]{inputenc}
\usepackage[T1]{fontenc}
\usepackage{amsmath}
\usepackage{amssymb}
\usepackage{amsthm,enumerate} 
\usepackage{listings}
\usepackage{empheq}
\usepackage{url}
\usepackage{hyperref}
\usepackage{cleveref}
\usepackage{verbatim}
\usepackage{dsfont}
\usepackage{xcolor}
\usepackage{booktabs}
\usepackage{aliascnt}
\usepackage{marvosym}

\newaliascnt{lemma}{theorem}
\newtheorem{lemma}[lemma]{Lemma}
\aliascntresetthe{lemma}
\crefname{lemma}{lemma}{lemmas}
\Crefname{lemma}{Lemma}{Lemmas}

\newaliascnt{conjecture}{theorem}

\aliascntresetthe{conjecture}
\crefname{conjecture}{conjecture}{conjectures}
\Crefname{conjecture}{Conjecture}{Conjectures}

\newaliascnt{example}{theorem}

\aliascntresetthe{example}
\crefname{example}{example}{examples}
\Crefname{example}{Example}{Examples}

\newaliascnt{cor}{theorem}
\newtheorem{cor}[cor]{Corollary}
\aliascntresetthe{cor}
\crefname{cor}{corollary}{corollaries}
\Crefname{cor}{Corollary}{Corollaries}

\newaliascnt{prop}{theorem}

\aliascntresetthe{prop}
\crefname{prop}{proposition}{propositions}
\Crefname{prop}{Proposition}{Propositions}

\newaliascnt{notation}{theorem}

\aliascntresetthe{notation}
\crefname{notation}{notation}{notations}
\Crefname{notation}{Notation}{Notations}

\newaliascnt{thm}{theorem}
\newtheorem{thm}[thm]{Theorem}
\aliascntresetthe{thm}
\crefname{thm}{theorem}{theorems}
\Crefname{thm}{Theorem}{Theorems}

\newaliascnt{remark}{theorem}

\aliascntresetthe{remark}
\crefname{remark}{remark}{remarks}
\Crefname{remark}{Remark}{Remarks}

\newaliascnt{result}{theorem}

\aliascntresetthe{result}
\crefname{result}{result}{results}
\Crefname{result}{Result}{Results}

\newaliascnt{obs}{theorem}

\aliascntresetthe{obs}
\crefname{obs}{observation}{observations}
\Crefname{obs}{Observation}{Observations}

\newaliascnt{problem}{theorem}

\aliascntresetthe{problem}
\crefname{problem}{problem}{problems}
\Crefname{problem}{Problem}{Problems}

\newaliascnt{de}{theorem}

\aliascntresetthe{de}
\crefname{de}{definition}{definitions}
\Crefname{de}{Definition}{Definitions}

\newcommand{\EKR}{\mathcal{F}}

\newcommand{\F}{\mathbb{F}}

\DeclareMathOperator{\PG}{PG}

\usepackage{cancel}

\title{A note on the chromatic number of Kneser graphs on chambers of projective planes and incidence-free sets}
\author{
Philipp Heering\footnote{Mathematisches Institut, Justus-Liebig Universität Gießen, Arndtstraße 2,
 D-35392 Gießen, Germany,
 (email: \href{mailto:philipp.heering@math.uni-giessen.de}{philipp.heering@math.uni-giessen.de})}, \
Klaus Metsch\footnote{Mathematisches Institut, Justus-Liebig Universität Gießen, Arndtstraße 2,
 D-35392 Gießen, Germany,
 (email: \href{mailto:klaus.metsch@math.uni-giessen.de}{klaus.metsch@math.uni-giessen.de})}, \
 Vladislav Taranchuk\footnote{Ghent University,
Department of Mathematics: Analysis, Logic and Discrete Mathematics, Krijgslaan 281,
B-9000 Gent, Belgium,
(email:  \href{mailto:vlad.taranchuk@ugent.be}{vlad.taranchuk@ugent.be})}, \ 
Zsuzsa Weiner\footnote{Department of Computer Science, E\"otv\"os Lor\'and University, P\'azm\'any P\'eter S\'et\'any 1/C, H-1117 Budapest, Hungary, (email: \href{mailto:zsuzsa.weiner@gmail.com}{zsuzsa.weiner@gmail.com})}
}

\date{2026}

\begin{document}

\maketitle

 \begin{center}
\subsection*{Abstract}
\end{center}

 Let $D=(\mathcal{P},\mathcal{B})$ be a symmetric $(v,k,\lambda)$-design and let $(X,Y)$ be an equinumerous incidence-free pair, with $X\subseteq \mathcal{P}$ and $Y\subseteq \mathcal{B}$. In this note, we give an elementary proof which shows the existence of a perfect matching between $\mathcal{P} \setminus X$ and $\mathcal{B}\setminus Y$ in the incidence graph of $D$. This recovers a result of Spiro, Adriaensen and Mattheus, who already showed this using different arguments for $k\geq 36$. We use this to connect some dots in the literature and prove that finding the chromatic number of the Kneser graph on chambers of a projective plane is equivalent to finding the incidence-free number of the incidence graph of the plane. Furthermore, we construct an incidence-free pair for PG$(2,q^2)$ of size roughly $q^3/2+3q^2/4$.

\section{Introduction}

The Kneser graph on flags of a projective space has the flags as vertices, with two vertices adjacent when the corresponding flags are opposite. Determining its maximal cocliques is therefore an Erd\H{o}s-Ko-Rado-type problem. For several choices of dimension and flag type, the largest cocliques are known, and geometric proofs often yield stability results. In some cases, these stability results can be leveraged to determine the chromatic number, see \cite{CN_of_some, DHAESELEER2022103474, heeringmetsch2023secondmax}.

All projective planes considered in this note are Desarguesian, even though we remark that for non-Desarguesian projective planes the following definition also works.
 We write $\Gamma_q$ for the graph whose vertices are the chambers (or point-line flags) $(P,\ell)$ of a projective plane of order $q$. Two vertices $(P_1,\ell_1)$ and $(P_2,\ell_2)$ are adjacent if and only if $P_1\notin \ell_2$ and $P_2\notin \ell_1$. 
 It was already observed in \cite{Mussche_PhD} that every non-empty, maximal coclique $\EKR$ of $\Gamma_q$ is one of the following:
 \begin{itemize}
     \item There is a chamber $(P,\ell)$ in the plane, such that all chambers of $\EKR$ have $\ell$ as their line, or $P$ as their point. Such a coclique $\EKR$ has size $2q+1$ and is called a chamber coclique.
     \item We have $|\EKR|=3$ and the points of the chambers are not collinear. Such a coclique is called a triangular coclique.
 \end{itemize}
 Thus, a stability result for the problem of maximal cocliques is available, but determining the chromatic number $\chi(\Gamma_q)$ remains non-trivial. It was shown in \cite{Mussche_PhD} that
\[
(q+1)(q+1-\sqrt{q})\leq \chi(\Gamma_q) \leq q^2+1,
\]
with equality on the left when a maximal arc of order $\sqrt{q}$ exists. Such arcs are known to exist in Desarguesian projective planes if $q$ is a power of four.

The point-line incidence structure of a projective plane is a symmetric $(q^2+q+1,q+1,1)$-design. We use the language of designs to state the main result of this note. Let $D=(\mathcal{P},\mathcal{B})$ be a symmetric $(v,k,\lambda)$-design. 
That means that $\mathcal{P}$ is a set of $v$ points and $\mathcal{B}$ is a set of $v$ blocks, each with $k$ points, such that any $2$ points occur in exactly $\lambda$ blocks.
Then, every point is on $k$ blocks and any two blocks share exactly $\lambda$ points.
If $X\subseteq \mathcal{P}$ and $Y\subseteq \mathcal{B}$ such that no block of $Y$ contains a point of $X$ and $|X|=|Y|$, then $(X,Y)$ is called an equinumerous incidence-free pair.

\begin{thm} \label{T: improve sam}
    Let $D=(\mathcal{P},\mathcal{B})$ be a symmetric $(v,k,\lambda)$-design. If $(X,Y)$ is an equinumerous incidence-free pair, with $X\subseteq \mathcal{P}$ and $Y\subseteq \mathcal{B}$, then there is a perfect matching between $\mathcal{P}\setminus X$ and $\mathcal{B}\setminus Y$ in the incidence graph of $D$.
\end{thm}

This generalizes a result due to S.A.M. [Proposition 5.1 in \cite{Sam}], 
where the same result was proven under the extra condition $k\geq 36$.
 Using Theorem \ref{T: improve sam} instead of Proposition 5.1 in \cite{Sam} also gives a generalization of their main result [Theorem 1.3 in \cite{Sam}] as we have no condition on $k$.

\begin{cor}
     Let $D=(\mathcal{P},\mathcal{B})$ be a symmetric $(v,k,\lambda)$-design. Then the edge domination number of this design is $v-\overline{\alpha}$, where $\overline{\alpha}$ is the maximum size of an equinumerous incidence-free pair $(X,Y)$.
\end{cor}

Our proof of Theorem \ref{T: improve sam} is quite elementary.
Furthermore, we can use it to prove the following.

\begin{thm} \label{T: connect the dots}
Finding a $t$-coloring of  the Kneser graph on chambers of a projective plane of order $q$ that uses only chamber cocliques is equivalent to finding an equinumerous, incidence-free pair $(X,Y)$ of points and lines of size $q^2+q+1-t$.
\end{thm}

We apply the following lemma from the thesis of Mussche.

\begin{lemma} [\cite{Mussche_PhD}]
    For every $t$-coloring of $\Gamma_q$ that contains a triangular coclique, there exists a coloring with at most $t$ colors that contains no triangular coclique.
\end{lemma}

This leads to the following corollary.

\begin{cor}
    Finding an optimal coloring of the Kneser graph on chambers of a projective plane is equivalent to finding a maximum, equinumerous, incidence-free pair $(X,Y)$ of points and lines in the plane.
\end{cor}

The size of $X$ in a largest possible, equinumerous, incidence-free pair $(X,Y)$ as in the theorem above is also called the incidence-free number of the incidence graph of the projective plane. Inspired by the notation of \cite{isoperimetric}, we denote this number as $\overline{\alpha}_q$. 
In \cite{Incidence-free_Sets_de_winter} a similar (but more general) number was studied.
In fact, the authors of \cite{Incidence-free_Sets_de_winter} reproved a result from Haemers, namely $\sqrt{q}(q-\sqrt{q}+1)\geq\overline{\alpha}_q$ with equality if and only if $q$ is a power of $4$. 
Theorem \ref{T: connect the dots} shows that this is equivalent to $(q+1)(q+1-\sqrt{q})\leq\chi(\Gamma_q)$ with equality if and only if $q$ is  a power of $4$.  This is the same lower bound that was also proven in \cite{Mussche_PhD}.

If $q=p^2$ is not a power of $4$, the currently largest known construction for $\overline{\alpha}_{p^2}$ is due to \cite{Incidence-free_Sets_de_winter} and gives $\frac{1}{2}p^3\leq \overline{\alpha}_{p^2}$. We improve this as follows.

\begin{thm}\label{T: improved unital construction2}
In $\PG(2,q^2)$,  there exists an incidence-free pair of size
\[
\frac{1}{2} q^3 + \frac{3}{4}q^2- \mathcal{O}(q).
\]

More precisely, let $r = \lfloor \frac{q-2}{2} \rfloor$, then we obtain an incidence-free pair of size at least
$\lfloor \frac{q^3 + q^2-qr}{2(q+1)}\rfloor (q+1)+ qr$.
\end{thm}

Below we list small values of $q$ for which $\chi(\Gamma_q)$ is now known. 
For $q=2,3$ and  powers of $4$ this was already determined in \cite{Mussche_PhD}, all other values utilize Theorem \ref{T: connect the dots} and \cite{isoperimetric, Ure_PhD}. 
In particular, Theorem 1.1 of \cite{isoperimetric}, in conjunction with Theorem \ref{T: connect the dots}, significantly improves the general upper bound $q^2+1$ of \cite{Mussche_PhD}. We remark that the more general bounds on $q$ in the table below require the plane to be Desarguesian.

\begin{table}[h]
    \centering
     \renewcommand{\arraystretch}{1.3}
    \begin{tabular}{|c|c|c|}
     \hline
       $q$ & $\chi(\Gamma_q)$ & Reference \\
     \hline\hline
    $2$ & $5$ & \cite{Mussche_PhD}  \\
    \hline
    $3$ & $10$ & \cite{Mussche_PhD} \\
    \hline
    $4$ & $15$ & \cite{Mussche_PhD}  \\
    \hline
    $5$ & $24$ & \cite{Ure_PhD}  \\
    \hline
    $7$ & $44$ & \cite{Ure_PhD} \\
    \hline
    $8$ & $57$ & \cite{isoperimetric} \\
    \hline
    $9$ & $72$ & \cite{isoperimetric} \\
    \hline
    $11$ & $122$ &  \cite{isoperimetric}\\
    \hline
    $13$ & $147$ & \cite{isoperimetric} \\
    \hline
    $p^2$ & $\leq p^4+p^2+1-(p^3/2+3p^2/4-\mathcal{O}(p))$ & \ref{T: improved unital construction2} \\
     \hline
    $q$ & $\leq q^2+q+1-(\frac{1}{2}q^{3/2} - \mathcal{O}(q^{5/4+\epsilon}))$ \text{\ for any $\epsilon>0$} & \cite{isoperimetric} \\
    \hline
    \end{tabular}
    \caption{Bounds on $\chi(\Gamma_q)$}
    \label{tab:placeholder}
\end{table}

Table \ref{tab:placeholder} uses $\overline{\alpha}_5=7$, so there is an incidence-free pair $(X,Y)$ of points and lines with $|X|=|Y|=7$ in a projective plane of order $5$.
When searching the literature, we could not find an explicit construction of such a set, we give one in Section \ref{section: constructions for incidence free number}.

\section{Proof of Theorem \ref{T: improve sam}}

Consider a bipartite graph with bipartite sets $A$ and $B$. An $A$-perfect matching is a matching of disjoint edges that covers every vertex in $A$. Assume that $W\subseteq A$ and denote by $N(W)$ the neighborhood of $W$ in $B$, i.e. the set of all vertices in $B$ that are adjacent to a vertex in $W$. If we have $|W|\leq|N(W)|$ for any such $W$, then Hall's marriage theorem implies the existence of an $A$-perfect matching. In particular, Hall's marriage theorem implies that any regular bipartite graph has a perfect matching.

The idea of our proof is as follows: 
First we fix a bijection $\Phi$ between the removed points $X$ and the removed blocks $Y$. 
For every pair $(x,\Phi(x))$ consider their neighbors.
In this neighborhood, we can find a (local) perfect matching.
These (local) perfect matchings can be used to construct a weight function for the edges, such that the assigned weights give a fractional perfect matching. This implies the existence of a perfect matching.

\begin{proof}[Proof of Theorem \ref{T: improve sam}]
  Let $D=(\mathcal{P},\mathcal{B})$ be a symmetric $(v,k,\lambda)$-design and let $(X,Y)$ be an equinumerous incidence-free pair, with $X\subseteq \mathcal{P}$ and $Y\subseteq \mathcal{B}$.

    Let $\Phi:X\rightarrow Y$ be a bijection. For each $x\in X$ define $N(x):= \{ C\in \mathcal{B} \mid x\in C \}$. As $(X,Y)$ is incidence-free, we have $N(x)\cap Y=\emptyset$.
    Furthermore, define $N(\Phi(x))=\{ p\in \mathcal{P} \mid p\in \Phi(x) \}$. As $\Phi(x)\in Y$ and $(X,Y)$ is incidence free, we have $N(\Phi(x))\cap X=\emptyset$.
    For $x\in X$ consider the bipartite graph $H_x$ with vertices $N(x)\cup N(\Phi(x))$, two vertices are adjacent if they are incident in $D$.
    If $C\in N(x)$, then $C\neq \Phi(x)$, as $N(X)\cap Y=\emptyset$, hence $|C\cap \Phi(x)|=\lambda$.
    For $p\in \Phi(x)$, there are $\lambda$ blocks that contain $p$ and $x$, all these blocks are in $N(x)$. 
    Hence, $H_x$ is a $\lambda$-regular, bipartite graph. 
    By Hall's marriage theorem, $H_x$ has a perfect matching, we denote it by $M_x$.

    Let $G$ be the point-block incidence graph of the design $D$ induced on the points $\mathcal{P}\setminus X$ and the blocks $\mathcal{B}\setminus Y$.
    As $N(x)\cap Y=\emptyset$ and $N(\Phi(x))\cap X=\emptyset$ for $x\in X$, the graph $H_x$ is a subgraph of $G$. In particular, all edges of the matching $M_x$ are also edges in $G$. 

    We define a weight $w(e)$ on each edge $e$ of $G$ via
    $$ w(e):=\frac{1+|\{ x\in X \mid e\in M_x \}|}{k}. $$
    The weight $w(v)$ of a vertex $v$ in $G$ is
    $$ w(v):=\sum\limits_{v\in e}w(e). $$
    Next up, we show that every vertex of $G$ in $\mathcal{P}\setminus X$ has weight $1$, this yields a $(\mathcal{P}\setminus X)$-perfect fractional matching.

    Let  $p\in \mathcal{P}\setminus X$. Define by $|N_Y(p)|$ the number of blocks incident with $p$ that are in $Y$. Since $p$ lies on $k$ blocks in the design $D$, it has degree $k-|N_Y(p)|$ in $G$.
    For each block in $Y$ that is incident with $p$, the bijection $\Phi$ gives us a unique point $x$ in $X$. By construction, we have that $p$ is a vertex of $H_x$. Hence there is precisely one edge in $M_x$ incident with $p$. As this is true for each block in $Y$ which is incident with $p$, the extra contribution summed up over all edges is $|N_Y(p)|$ and the total weight is $w(p)=\frac{k-|N_Y(p)|+|N_Y(p)|}{k}=1$.

The fact that a $(\mathcal{P}\setminus X)$-perfect fractional matching implies the existence of a  $(\mathcal{P}\setminus X)$-perfect matching is well known, we provide the short argument: 
By a completely symmetric argument, we obtain that every vertex $C$ of $G$ that is in $\mathcal{B}\setminus Y$ also has weight $1$.
Let $S\subseteq \mathcal{P}\setminus X$.
Denote by $N(S)$ the set of all blocks in $\mathcal{B}\setminus Y$ that are incident with a point in $S$.
Then
\begin{align*}
    |S|&=\sum\limits_{p\in S}1=\sum\limits_{p\in S} 
    \sum\limits_{\substack{C\in \mathcal{B}\setminus Y \\ p\in C}} w(\{p,C\})=   
    \sum\limits_{C\in N(S)} 
    \sum\limits_{\substack{p\in \mathcal{P}\setminus X \\ p\in S\cap C}} w(\{p,C\})\\
    &\leq  \sum\limits_{C\in N(S)} \sum\limits_{\substack{p\in \mathcal{P}\setminus X \\ p\in C}}w(\{p,C\})=\sum\limits_{C\in N(S)}1=|N(S)|.  
\end{align*}
    Hence $|S|\leq |N(S)|$ and Hall's marriage theorem yields that $G$ has a $(\mathcal{P}\setminus X)$-perfect matching.
    As $|\mathcal{P}\setminus X|=|\mathcal{B}\setminus Y|$, this is a perfect matching of $G$.
\end{proof}

\section{Proof of Theorem \ref{T: connect the dots}}
\label{Section: 3}

\begin{proof}[Proof of Theorem \ref{T: connect the dots}]
First, we show how we find an incidence-free, equinumerous pair $(X,Y)$ of points and lines, starting from a coloring that uses only chamber cocliques. 
We call a coloring redundant if some color can be removed entirely, with the vertices of that color recolored using only the remaining colors, so that the resulting coloring is still valid.
If our coloring is redundant, this translates to finding an equinumerous incidence-free pair with fewer elements. So we can assume that our $t$-coloring is not redundant.

Assume that we have such a $t$-coloring. Then, this coloring can be described as follows: There are chambers $(P_1,\ell_1),\ldots,(P_t,\ell_t)$, such that for every chamber $(P,\ell)$ there is an $i$ such that $P=P_i$ or $\ell=\ell_i$. Now, assume that the  points $P_i$ and lines $\ell_i$ are not all pairwise distinct. Without loss of generality let $P_1=P_2$ and $\ell_1\neq \ell_2$. There are $2q+1$ chambers in the coclique based on $(P_1,\ell_1)$ and there are only $q$ other chambers (namely $(P_2^i,\ell_2)$ for $P_2^i\neq P_2$) in the coclique based on $(P_2,\ell_2)$. These $q$ chambers $(P_2^i,\ell_2)$ are in every coclique based on a chamber $(P_2',\ell_2)$ with $P_2'\in\ell_2$. So if we can change $P_2$ to a point on $\ell_2$ and not in $P_1,P_3,\ldots,P_t$, then our list of chambers has the desired property. If we cannot make such a change, that is if every point of $\ell_2$ occurs in the list $P_1,P_3,\ldots,P_t$, then the $q$ chambers $(P_2^i,\ell_2)$ of the coclique based on $(P_2,\ell_2)$ that are not in the coclique based on $(P_1,\ell_1)$, are all contained in the cocliques based on $(P_3,\ell_3),\ldots,(P_t,\ell_t)$. In this case our coloring was redundant, which is a contradiction.

Denote by $\mathcal{P}$ the set of points in the plane and let $X:=\mathcal{P}\setminus\{ P_1,\ldots ,P_t\}$. Similarly denote by $\mathcal{L}$  the set of lines in the plane and let $Y:=\mathcal{L}\setminus \{ \ell_1,\ldots ,\ell_t\}$. Now $(X,Y)$ is an equinumerous pair with $|X|=|Y|=q^2+q+1-t$. Assume that $(X,Y)$ is not incidence-free. Then there is a point $P\notin \{ P_1,\ldots ,P_t\}$ and a line $\ell\notin \{ \ell_1,\ldots ,\ell_t\}$, such that $(P, \ell)$ is a chamber. Then there is no $i$, such that $P=P_i$ or $\ell=\ell_i$, which is a contradiction.

Now, we show how to find a coloring, starting from an incidence-free, equinumerous pair $(X,Y)$ of points and lines.

Consider the incidence graph $G$ of the projective plane that we obtain if we remove $X\cup Y$.
By Theorem \ref{T: improve sam}, we know that we can find a perfect matching of this graph. 
Every edge of the perfect matching connects a point to an incident line. So, we can understand the matching as a list of chambers $(P_1,\ell_1),\ldots,(P_t,\ell_t)$.
We claim that for every chamber $(P,\ell)$ there exists an $i$, such that $P=P_i$ or $\ell=\ell_i$. If there is no $P_i$ with $P_i=P$, then we must have $P\in X$. But since $\ell$ is incident with $P$ and $(X,Y)$ is incidence-free, then $\ell\notin Y$. Hence, $\ell$ is a vertex in $G$, so there is an $i$ with $\ell=\ell_i$.

Now, every vertex of $\Gamma_q$ is in at least one of the cocliques determined by the chambers $(P_1,\ell_1),\ldots,(P_t,\ell_t)$. This gives a coloring with $t$ colors.
\end{proof}

\section{Geometric constructions of incidence-free pairs}
\label{section: constructions for incidence free number}

First, we geometrically describe an incidence-free pair $(X,Y)$ of points and lines with $|X|=|Y|=7$ for $q=5$, which is best possible via \cite{isoperimetric}. Our construction resembles a grid.

Fix a point $P$ and consider the lines $\ell_1,\ldots,\ell_{6}$ incident with $P$.
Next, fix a line $g$ that is not incident with $P$.
Denote by $Q$ the point $g\cap \ell_1$. Let $g_1, g_2$ be distinct lines that are incident with $Q$, but not $P$, such that $g\neq g_1,g_2$.
Let $g_3$ be the line that is incident with $g_1\cap \ell_2$ and $g_2\cap \ell_3$.
We define $Y:=\{ \ell_4,\ell_5,\ell_6,g,g_1,g_2,g_3 \}$.

Let $X$ be the set of points on $\ell_1,\ell_2,\ell_3$ that are not incident with any of the lines in $Y$.
The lines $\ell_4,\ell_5,\ell_6$ intersect $\ell_1,\ell_2,\ell_3$ only in $P$.
Furthermore, the line $g$ intersects $\ell_1,\ell_2,\ell_3$ in three distinct points.
So, there are $q-1=4$ points left on each of the three lines.

The lines $g_1,g_2$ intersect $\ell_1$ in the same point as $g$. However, $g_3$ intersects $\ell_1$ in a different point. Hence, there are $3$ points left on $\ell_1$.

The lines $g_1,g_2$ intersect $\ell_2,\ell_3$ in points not on $g$.
However, $g_3$ intersects $\ell_2,\ell_3$ in $g_1\cap \ell_2$ and $g_2\cap \ell_3$, so it removes no additional points. Hence, there are $2$ points left on $\ell_2,\ell_3$.
In total, we have $|X|=3+2+2=7$.\\

Next, we provide a construction for a large incidence-free pair of $\PG(2,q^2)$. For this construction we need unitals. A unital of $\PG(2,q^2)$ is a set of $q^3+1$ points such that every line of the plane contains exactly one, or $q+1$ points of the unital. Roughly speaking, the idea in \cite{Incidence-free_Sets_de_winter} is to take half of the points of the unital and the tangent lines for the other half. However, the authors of \cite{Incidence-free_Sets_de_winter} did already remark that choosing the points more carefully can improve the bound.

\begin{proof}[Proof of Theorem \ref{T: improved unital construction2}]

Let $P_{\infty}$ be the ideal point of the Hermitian unital $U$ in $\PG(2, q^2)$ and let us write the affine part of $U$ in the norm-trace form, i.e.
\[
U \setminus \{P_{\infty}\} = \{(x, y): N(x) = \operatorname{Tr}(y)\}, 
\]
where $N(x) = x^{q+1}$ and $\operatorname{Tr}(y) = y^q+y$, for details of this representation we refer to Section 12.3 of \cite{hermitian_unital_norm-trace}. Hence the tangent at $P_{\infty}$ is the line at infinity. First of all, note that for each $c \in \F_q^*$ the set $G_c=\{(x,y):N(x)=c,\ \operatorname{Tr}(y)=c\}$ is contained in $U$. The equation $N(x)=c$ has $q+1$ solutions, while $\operatorname{Tr}(y)=c$ has $q$ solutions; so $G_c$ is a $(q+1)$ by $q$ grid of unital points. Choose a set $C \subset \F_q^*$ of size $r = \left\lfloor\frac{q-2} {2}\right\rfloor$ and let $A_0 = \{0\} \cup N^{-1}(C)$, that is \[A_0 := \{0\} \cup \{ a \in \F_{q^2}: N(a) \in C\},\] hence $|A_0| = 1+r(q+1)$. Now choose a set $A \subset \F_{q^2}$ of size $\left\lfloor\frac{q^3+q^2-rq}{2(q+1)}\right\rfloor$ which contains $A_0$. 

Now we construct our incidence-free structure $(X, Y)$. The point set $X$ consists of the affine points $(x, y)$ of $U$ such that $x \not \in A$, together with certain points on the line at infinity. We add the points on the line at infinity except $P_\infty$ and except the points in which the tangents at the deleted affine unital points (with $x\in A$) meet the line at infinity. The line set $Y$ consists of all tangents at the deleted affine unital points, the vertical lines $x=a$, where $a \in A$ and the horizontal lines with $y=b$ where $\operatorname{Tr}(b) \in C$.

The tangent lines in $Y$ are disjoint from $X$, since each of them meets $U$ only in its deleted affine point and meets the line at infinity in one of the excluded points. The chosen vertical and horizontal lines are also disjoint from $X$: the vertical lines contain - apart from $P_\infty$ - only deleted affine unital points, while if $\operatorname{Tr}(b)=c\in C$, then the affine unital points on the horizontal line $Y=bZ$ have $N(x)=c$, hence $x\in A_0\subseteq A$, and the point at infinity of these horizontal lines is excluded because $0\in A$.

Note that $Y$ contains $|A|q$ tangent lines (of $U$), $|A|$ vertical lines and $rq$ horizontal lines; hence \[|Y| = |A|(q+1)+rq. \] Also, $X$ contains $q^3-|A|q$ affine points from the unital $U$ and $q^2-|A|$ points from the line at infinity; hence \[|X|= q^3+ q^2-|A|q-|A|.\] By our choice of the size of $A$, we have $|Y| \le |X|$ and so by removing some extra points from $X$ if needed, we get an equinumerous incidence-free structure of size $\left\lfloor \frac{q^3+q^2-qr}{2(q+1)}\right\rfloor (q+1)+qr$ and hence the theorem follows.
\end{proof}

\subsection*{AI tool disclosure}
The proof of Theorem \ref{T: improve sam} is based on a suggestion from ChatGPT Pro. The final arguments in this paper are human-generated.

\subsection*{Acknowledgements}
The first author would like to acknowledge the support of the Universiteit Gent, where part of this work was done while the first author was a visitor. 
Theorem \ref{T: improved unital construction2} was conceptualized during the conference Combinatorics 2026 in Naples, the authors would like to thank the organizers for a wonderful conference.

\bibliographystyle{plainurl}
\bibliography{bib_CN.bib}

\end{document}